\documentclass{amsart}

\usepackage{amsfonts,amsthm,latexsym,amsmath,amssymb,amscd,amsmath, mathrsfs, epsf,amsmath}
\newtheorem{theorem}{Theorem}

\newtheorem{lemma}[theorem]{Lemma}

\newtheorem{remark}{Remark}

\newcommand{\be}{\begin{equation}}
\newcommand{\ee}{\end{equation}}

\DeclareMathOperator{\arccosh}{arccosh}

\begin{document}

\title[On approximation of ultraspherical polynomials...]{On approximation of ultraspherical polynomials in the oscillatory region}

\author[I. Krasikov]{Ilia Krasikov}


\address{ Department of Mathematics, Brunel University London, Uxbridge
UB8 3PH United Kingdom} \email{mastiik@brunel.ac.uk}

\begin{abstract}
For $k \ge 2$ even, and $ \alpha \ge -(2k+1)/4 $, we provide a uniform approximation of the ultraspherical polynomials $ P_k^{(\alpha,\, \alpha)}(x) $ in the oscillatory region with a very explicit error term. In fact, our result covers all $\alpha$ for which the expression ``oscillatory region" makes sense. We show that there the function $g(x)={c \sqrt{b(x)} \, (1-x^2)^{(\alpha+1)/2} P_k^{(\alpha, \alpha)}(x)=\cos \mathcal{B}(x)+ r(x)}$, where $c=c(k, \alpha)$ is defined by the normalization, $\mathcal{B}(x)=\int_{0}^ x b(x) dx$, and the functions $c,\; b(x), \; \mathcal{B}(x)$, as well as bounds on the error term $r(x)$ are given by some rather simple elementary functions.
\end{abstract}

\maketitle

\noindent
\hspace{1ex}{\bf Keywords:}
orthogonal polynomials, ultraspherical polynomials, Gegenbauer polynomials, uniform approximation
\noindent

\section{Introduction}
The ultraspherical polynomials we deal with in this paper will be convenient to define in terms of Jacobi polynomials as $ P_k^{(\alpha,\, \alpha)}(x)$, where we choose the standard normalization for the last function. We will use the bold character ${\bf P}_k^{(\alpha,\, \alpha)}(x) $ to denote the orthonormal Jacobi polynomials.
Since we are going to consider the case $\alpha \le -1$ as well, let us notice that the right hand side of the formula
$$\big|\big|P_k^{(\alpha, \, \alpha)}\big|\big|^2_{L_2}=\int_{-1}^1 (1-x^2)^\alpha \left( P_k^{(\alpha,\, \alpha)}(x) \right)^2 dx = \dfrac{2^{2\alpha+1} \,\Gamma^2(k+\alpha+1)}{(2k+2\alpha+1)\, \Gamma^2(k+2\alpha+1)k!} \, ,$$
and therefore the orthonormal normalization, still make sense as far as $\alpha > - \dfrac{k+1}{2} \, $.

 Here we will establish a uniform approximation of the ultraspherical polynomials in the oscillatory region with an explicit error term for a vast range of the parameter $\alpha$; in fact, for all $\alpha$ for which the expression ``oscillatory region" makes sense.
  A few standard formulas we are using in the sequel may be found e.g. in \cite{szego}.

 There are a number of known asymptotics for the Jacobi polynomials under these or those restrictions on the parameters $\alpha$ and $\beta$, starting from the classical case $|\alpha |, |\beta| \le 1/2$ considered in Szeg\"{o}'s book \cite{szego}, or, for example, more recent results with asymptotically constant ratios of $\alpha/k$ and  $\beta/k$ (see e.g. \cite{KMF}, \cite{OSRZ} and references therein). However if one is interested in uniform bounds the situation becomes less studied, and we refer to the recent preprint \cite{KKT} and the references therein for a review of known results.

To simplify otherwise complicated expressions in the sequel we introduce the following new parameters:
\begin{equation}
\label{vu}
 u=(k+ \alpha)( k+\alpha+1), \; \; q=(\alpha^2-1)/u  ,
  \end{equation}
which turn out to be quite natural in this context.

 We start with the normal form of the differential equation for ultraspherical polynomials
 \begin{equation}
\label{maindifeq}
y''+b^2 y=0, \; \; \; y=(1-x^2)^{(\alpha+1)/2} P_k^{(\alpha, \, \alpha)}(x),
\end{equation}
where
\begin{equation}
\label{defb}
b=b(x)=\frac{\sqrt{(1-q-x^2)u}}{1-x^2} \, ,
\end{equation}
 and define the function
 \begin{equation}
\label{defg}
g(x)= \sqrt{b(x)} \, y(x) ,
\end{equation}
that, as we will show, almost  equioscillates in the interval $|x| <\sqrt{1-q} \, $.

Our main result is the following Theorem \ref{thmain1} that provides a uniform approximation of $g(x)$ for $k$ even in the oscillatory region with a very explicit error term.
The corresponding result for $k$ odd may be readily obtained from e.g. the three term recurrence. To simplify the statement of the theorem involving multivalued functions, without loss of generality we will restrict ourselves to the case $x \ge 0$.
\begin{theorem}
\label{thmain1} Let $k \ge 2$ be even and let $x$ belong to one of the following intervals depending on the value of $\alpha :$
\begin{enumerate}
\item[(\it{i})]
$ \; \;\; 0 \le x \le \sqrt{1-\dfrac{1}{u}} \, , \; \; \; \; | \alpha| < \sqrt{\dfrac{7}{6}} \, ; $
\item[(\it{ii})]
$ \; \;\; 0 \le x \le \sqrt{1-q} \, , \; \;\; \;  \alpha \in \big[ -\dfrac{2k+1}{4}\, , -\sqrt{\dfrac{7}{6}}\; \big] \, \bigcup \, \big[ \,\sqrt{\dfrac{7}{6}} \, , \infty \, \big).$
\end{enumerate}
Then the following approximation holds:
\begin{equation}
\label{eqg}
g(x)=g(0) \big( \cos \mathcal{B}(x)+r(x) \big),
\end{equation}
where
\begin{equation}
\label{urg0}
g(0)=
 \left( - \frac{1}{4} \right)^{k/2} {k+\alpha \choose k/2} \left( k^2+2k \alpha+k+\alpha+1 \right)^{1/4};
\end{equation}
\begin{equation}
\label{urbb}
\mathcal{B}(x) =
\end{equation}
 $$\left\{
\begin{array}{cc}
 \sqrt{u} \left( \arccos \sqrt{\dfrac{1-q-x^2}{1-q}}-\sqrt{q} \, \arccos \sqrt{\dfrac{1-q-x^2}{(1-q)(1-x^2)}} \; \,\right),  & q \ge 0;\\
 &\\
 \sqrt{u} \left( \arccos \sqrt{\dfrac{1-q-x^2}{1-q}}+\sqrt{-q} \,\arccosh \, \sqrt{\dfrac{1-q-x^2}{(1-q)(1-x^2)}} \;\, \right), & q < 0.
 \end{array}
 \right.$$

The error term $r(x)$ is bounded as follows:
\begin{equation}
\label{urrx}
 |r(x)| < \left\{
\begin{array}{cc}
\dfrac{2.72 \, x}{\sqrt{(1-x^2)u}} \, , & q < 0 ;\\
   &\\
 \dfrac{2(1-x^2)x}{(1-q)(1-q-x^2)^{3/2} \, \sqrt{u}}\, , &  0 \le q <\frac{1}{2} \, ;\\
&\\
\dfrac{(1+q)x}{4\,(1 - q - x^2)^{3/2}\sqrt{u}} \, ,&\frac{1}{2} \le q <1 .
\end{array}
\right.
\end{equation}

\end{theorem}
As a corollary we deduce that the ultraspherical polynomials (in a sense also for $\alpha \le -1$), live in the interval $\;(\,-\sqrt{1-q} \,,\,\sqrt{1-q} \;)$; more precisely
\begin{theorem}
\label{sharp} For $k \ge 10$ even,  $\alpha \ge 3-\dfrac{k}{2} \, , \; \; |\alpha| \ge 1$,
$$\int_{-\eta}^{\eta} (1-x^2)^\alpha \left( {\bf P}_k^{(\alpha, \alpha)} (x) \right)^2 dx  >1- \frac{5}{3(1-q)^{1/3} u^{1/6}} > 1-\frac{5}{3} \cdot \left( \dfrac{k+\alpha}{(k+2\alpha)k}\right)^{1/3},$$
where
$$\eta =\sqrt{1-q} \, \left(1-\dfrac{4 \cdot 2^{1/3}}{3(1-q)^{2/3} u^{1/3}} \, \right).$$
\end{theorem}

Let us make a few comments about the theorems.

\begin{remark}
In fact, the first formula for $\mathcal{B}(x)$ holds also for $q <0$, provided one chooses the principal branches of the square roots and arccosines.
The expression for $\mathcal{B}(x)$ can be slightly simplified by the substitution $x=\sqrt{1-q} \, \sin \phi$, yielding
$$\mathcal{B}(x) = \sqrt{u} \, \left( \phi-\sqrt{q} \, \arccos \dfrac{1}{\sqrt{1+q \tan^2 \phi}} \right), \; \; \; 0 \le \phi <\frac{\pi}{2} \,.$$
\end{remark}
\begin{remark} For $k \ge 2$ and $\alpha \ge -(2k+1)/4$ the parameter $q$ satisfies the inequalities
\begin{equation}
\label{minq}
4 \sqrt{6}-10 \le - \frac{2}{k^2+k-1+\sqrt{(k-1)k(k+1)(k+2)}} \le q < 1 .
\end{equation}
Let us also note that $q <0$ corresponds to $|\alpha|<1$, and $q=\frac{1}{2}$  to
${\alpha \approx (1 \pm \sqrt{2})(k+\frac{1}{2} )}$.
\end{remark}
\begin{remark}
For orthonormal Jacobi polynomials Stirling's approximation gives
 $$g(0) \approx (-1)^{k/2} \sqrt{\frac{2k+2\alpha+1}{\pi}}  \,.$$
\end{remark}

\begin{remark}
For $\alpha \ge \sqrt{7/6} \,$ the interval $|x| < \sqrt{1-q}$ is large enough to include all the zeros of  $ P_k^{(\alpha,\, \alpha)}(x) .$  We will show that even at the extreme zeros the error term is still of order $O(1)$ (see Remark \ref{clxmax} below). It seems not too much is known about the zeros of $P_k^{(\alpha , \alpha)} (x)$ for $\alpha <-1$ (see however \cite{DD}, \cite{DM}). Nevertheless, for negative $\alpha$ Theorem \ref{thmain1} covers practically the whole oscillatory region inside the interval $[-1,1]$ besides maybe extreme zeros. Indeed, it's easy to see that for a continuous function $a(x)$, a nontrivial solution $f(x)$ of the differential equation $f''+a f=0$ may have at most one zero in each interval where $a(x)<0$.
For $\alpha=-(2k+1)/4$ the length of interval $|x| < \sqrt{1-q}\,$ shrinks to approximately $\sqrt{3}/k$, however
 for $\alpha \le -k/2$ the corresponding ultraspherical polynomial has no zeros in the interval $|x|<1$ (see \cite[sec. 6.72 ]{szego}).
\end{remark}


\section{The main term}
We will use a version of WKB method presented in \cite{krasbes1}. The derivation of the approximation (\ref{eqg}) is quite straightforward, and some rather technical work is needed to
estimate the error term $r(x)$ only.

The function $g(x)$ satisfies the following differential equation
\begin{equation}
\label{difg}
g''-\frac{b'}{b} g' +(1+\varepsilon)b^2 g=0,
\end{equation}
where
\begin{equation}
\label{defepsilon}
\varepsilon=\varepsilon (x)= \frac{3b'^2-2b b''}{4b^4} =- \frac{x^6+6q x^4-3x^2+4q^2-6q+2}{4u(1-q-x^2)^3}\, .
\end{equation}
Solving this equation as inhomogeneous one with the right hand side $-\varepsilon b^2 g ,$ we obtain
\begin{equation}
\label{resh}
g(x)=M \cos (\mathcal{B}(x)+\gamma)+ R(x),
\end{equation}
where
\begin{equation}
\label{bigb}
\mathcal{B}(x)=\int_0^x b(t) dt ,
\end{equation}
is given explicitly by (\ref{urbb}), and
\begin{equation}
\label{urr}
R(x)=\int_0^x \varepsilon(t) b(t) g(t) \sin \big[ \mathcal{B}(x)-\mathcal{B}(t) \big]  dt .
\end{equation}
In the case of ultraspherical polynomials (\ref{resh}) the constants of integration $M$ and $\gamma$ can be readily found, and
we obtain the following claim, which is the first (and easy) part of Theorem \ref{thmain1}.
 \begin{lemma}
 \label{leoz}
 For $k$ even,
\begin{equation}
\label{solg}
g(x)=g(0) \left(\cos \mathcal{B}(x)+ r(x) \right),
\end{equation}
where
$r(x)=R(x)/g(0)$, and
$$
g(0)=u^{1/4}(1-q)^{1/4} \,  P_k^{(\alpha,\, \alpha)}(0) =
 \left( - \frac{1}{4} \right)^{k/2} {k+\alpha \choose k/2} \left( k^2+2k \alpha+k+\alpha+1 \right)^{1/4} .$$
 \end{lemma}
\begin{proof}
Plugging $x=0$ into (\ref{resh}) yields
$M=g(0) \cos \gamma $. Here the constant $\gamma$ must vanish since by (\ref{defg})
$$\left. g'(0)=(1-q)^{1/4} u^{1/4} \frac{d}{dx} P_k^{(\alpha , \alpha)}(x) \right|_{x=0} =0, $$
for $k$ even, whereas (\ref{solg}) gives
$$g'(0)=-M \sqrt{(1-q) u} \, \sin{(\mathcal{B}(0)+\gamma)}+ R'(0) =-M \sqrt{(1-q) u} \,\sin \gamma .$$
\end{proof}

\section{The error term}
In what follows we will assume that $k$ is even, and, whenever it is convenient, that $x \ge 0$.

For $x$ belonging to an interval $\mathcal{I}$ we will estimate the error term in the following straightforward manner:
\begin{equation}
\label{rrr}
|r(x)|= \left|\, \frac{1}{g(0)} \,\int_0^x \varepsilon(t) b(t) g(t) \sin \left[ \mathcal{B}(t)-\mathcal{B}(x) \right]  dt  \right| \le \mu \int_0^x |\varepsilon(t) b(t)| dt,
\end{equation}
where
\begin{equation}
\label{mu1}
\mu=\mu(k, \alpha)=\sup\limits_{x \in \mathcal{I}} \left|\frac{g(x)}{g(0)}\, \right|.
\end{equation}
\noindent
To estimate  $\sup\limits_{x \in \mathcal{I}} |g(x)|$ we consider the envelope of $g^2(x)$ given by
Sonin's function $S(x)$,
\begin{equation}
\label{sonin}
S(x)=g^2(x)+\dfrac{g'^2(x)}{\big(1+\varepsilon(x) \big) b^2(x)}= g^2(x)+\dfrac{4b^2(x)}{{\bf A}_1(x)} \,g'^2(x) ,
\end{equation}
where
$$  {\bf A}_1={\bf A}_1(x)=4 b^4(1+\varepsilon )=4b^4+3b'^2-2b b''.$$
Hence
\begin{equation}
\label{gs}
g^2(x) \le S(x),
\end{equation}
as far as ${\bf A}_1 (x) \ge 0 $.
The location of the maximum of $S(x)$ depends on the function $b(x)$ only. Indeed,
differentiating $S(x)$ and using (\ref{difg}) to get rid of $g''$, we get
$$S'= \frac{8b {\bf A}_2}{{\bf A}_1^2} \, g'^2 , \; \; \; {\bf A}_2={\bf A}_2(x)=6b'^3-6b b' b''+b^2 b'''.$$
Notice also that for $k$ even the point $x=0$ is a local maximum of $g(x)$, hence
\begin{equation}
\label{urm}
g^2(0)=S(0).
\end{equation}

Since we are mainly interested in the sign of ${\bf A}_1$ and ${\bf A}_2 $, it will be more convenient to deal with the following two polynomials instead:
\begin{equation}
\label{aaa1}
A_1=A_1(x)=\frac{(1-q-x^2)(1-x^2)^4}{u} \,{\bf A}_1 =
\end{equation}
$$4u(1-q-x^2)^3 +3x^2(1-6q x^2-x^4)-2(1-q)(1-2q),$$
and
\begin{equation}
\label{aaa2}
A_2=A_2(x)= \frac{(1-x^2)^4(1-q-x^2)^{3/2}}{3x \, u^{3/2}}\,{\bf A}_2=(1-q)(1-4q)-(1+q)x^2.
\end{equation}

It will also be convenient to introduce the following notion.
We call a multivariable polynomial $p(x), \; x \in \mathbb{R}^n,$ a $\mathcal{P}-$polynomial if its coefficients are nonnegative and it has a positive free term.
The only property of the $\mathcal{P}-$polynomials we use in the sequel is that $p(x)>0$ in the nonnegative orthant $\mathbb{R}_+^n$.

In general, the domain $\{x:A_1(x) >0 \}$ depends on $\alpha$ and $k$. As well, to have $ b(x) >0$ on a nonempty interval one needs $\alpha > -\frac{k^2+k+1}{2k+1} \, $. In what follows we impose a slightly stronger constraint, namely $\alpha \ge -(2k+1)/4$.
\begin{lemma}
\label{a1}
Let $k \ge 2$, then $A_1(x) >0$ in the following two cases:
$$
\begin{array}{ccc}
(i) &0 \le  x < \sqrt{1-q} \, , &  |\alpha| \ge \dfrac{\sqrt{10}}{3} \, ;\\
&&\\
 (ii) &0 \le  x \le \sqrt{1-\dfrac{1}{u}} \, , & |\alpha| <\dfrac{\sqrt{10}}{3} \, .
\end{array}
$$
\end{lemma}
 \begin{proof}
 To establish the first case we substitute $x=\sqrt{(1-q)\frac{t^2}{1+t^2} } \, $ into $A_1$ yielding
 $$ A_{11}(t,q)=\frac{(1+t^2)^3}{1-q} \, A_1=5q^2 t^6+6 q (1 + q) t^4-3 (1 - 4 q) t^2+4(1-q)^2 u+4q-2.$$
 Then $\min\limits_q A_{11}(0,q)=2-1/u >0$.
 For $q > 0, \; t > 0$, the function $A_{11}(t,q)$ has the only minimum at $t=t_0= \sqrt{\dfrac{1-4q}{5q}} \, ,$ with
 $$A_{12}=\frac{25q}{4} \, A_{11}(t_0,q)=25(1-q)^2 (\alpha^2-1)+8q^3-11q^2+4q-2. $$
 Thus, $A_1(x) > 0$ for $ q \ge 1/4$.
 Let now $ q < 1/4$, we
  substitute $q=\dfrac{s^2}{4(1+s^2)} \,,  \; s >0$, and ${\alpha^2=\dfrac{10+\delta^2}{9} }\, $ into $A_{12}$ getting
 $$A_{12} =\frac{150s^4+280s^2+112+25(s^2+1)(3s^2+4)^2 \delta^2}{144(1+s^2)^3}  > 0.$$
 This proves $(i)$.

It will be convenient to prove $(ii)$ for a slightly lager interval $|\alpha| \le 5/4$.
To demonstrate the inequality $A_1(x) >0$ one shows first that
$$\frac{1}{6x}\, \frac{d A_1}{d x}= 1 - 4 q x^2 - x^4-4 u (1 - q - x^2)^2 <0,$$  for
$0 \le x  \le \sqrt{1-1/u} \,.$
This can be done with the help of the substitutions
$$x=\frac{\sqrt{1-1/u}}{\sqrt{1+t^2}}\, , \; \; \alpha =\frac{5}{4} \cdot \frac{1-s^2}{1+s^2} \,, \; \; k=\kappa+2,$$
yielding the $\mathcal{P}-$polynomial
$$-(1+s^2)^6 (1+t^2)^2 u^2 \cdot \frac{1}{6x}\, \frac{d A_1}{d x}\, .$$
Hence in the second case
$$u^3 A_1(x) \ge u^3 A_1 (\sqrt{1-1/u} \;)= \big(1+4(1-\alpha^2)^2(3-\alpha^2) \big)u+7-6 \alpha^2 \ge
u+7-6 \alpha^2 >0.$$
  This completes the proof.
\end{proof}
We need one more technical claim.
\begin{lemma}
\label{maxs}
$$\sqrt{1-\dfrac{1}{u}} < \sqrt{\dfrac{(1-q)(1-4q)}{1+q}} \, ,$$
provided $k \ge 2$ and $|\alpha|\le \sqrt{7/6}$.
\end{lemma}
\begin{proof} For $|\alpha |\le 1$ the claim is obvious.
Noticing that $q\le \dfrac{1}{6u}$ for $1 < |\alpha| \le \sqrt{7/6} \, ,$ one finds
$$\dfrac{(1-q)(1-4q)}{1+q}\, - 1+\frac{1}{u} = \frac{4q^2 u-6 q u+q+1}{(1+q) u}\ge \frac{5}{18(1+q)u^2}>0.$$

 \end{proof}

By (\ref{aaa2}) the maximum of $S(x)$ is attained at $x=0$ for $q \ge 1/4$, or, in terms of $\alpha$, for $\alpha \notin (\alpha^- , \alpha^+)$, where
\begin{equation}
\label{alpha+-}
\alpha^\pm=\frac{2 k+1 \pm \sqrt{16k^2+16 k+49}}{6} \, .
\end{equation}
For $q <1/4$ the maximum of $S(x)$ is attained either at $x=x_0$, where
\begin{equation}
\label{x0}
x_0=\sqrt{\dfrac{(1-q)(1-4q)}{1+q}} \, ,
\end{equation}
or at the endpoint $x=\sqrt{1-1/u}\;$,  if $|\alpha| <\sqrt{10}/3$, by Lemmas \ref{a1} and \ref{maxs}.
To simplify the statement of the results we use Lemma \ref{maxs} to restrict the values of $x$ to $0 \le x \le \sqrt{1-1/u}\;$ in a slightly longer interval of the values of $\alpha$, $|\alpha| <\sqrt{7/6}\,$. As it is easy to check this still implies $q <1/4$.

The following simple lemma enables one to bound the value of $S(x)$ for $q <1/4$.

\begin{lemma}
\label{sx0}
Let $A_1(t)>0, \; A_2(t) >0$ for $0 \le t \le x$, then
\begin{equation}
g^2(x) \le S(x) \le \frac{1+\varepsilon(0)}{1+\varepsilon(x)} \, S(0) =\frac{1+\varepsilon(0)}{1+\varepsilon(x)} \, g^2(0).
\end{equation}
\end{lemma}
\begin{proof}
Starting with the identity
$$S- \frac{b A_1}{2 A_2} \, S'=g^2 \ge 0, $$
one obtains
$$ S'/S \le \frac{2 {\bf A}_2}{b {\bf A}_1} = \frac{d}{dx} \ln \frac{1}{4(1+\varepsilon(x))} \, ,$$
where $1+\varepsilon(x) >0$ by $A_1 >0$.
Integrating from 0 to $x$ we find
$$S(x)/S(0) \le \frac{1+\varepsilon (0)} {1+\varepsilon (x)} \,,$$
and the result follows by (\ref{urm}).
\end{proof}

Now we are in the position to find the factor $\mu$.

\begin{lemma}
\label{mu}
\begin{equation}
\label{eqmu}
\mu \le \left\{
\begin{array}{cc}
1,& \alpha \in [-\dfrac{2k+1}{4} \, , \alpha^-\,]\cup[\alpha^+, \infty\,);\\
&\\
\sqrt{\dfrac{\alpha^2-1}{ \alpha^2-1.08}} \, ,& \alpha \in \big( \alpha^-, -\sqrt{\dfrac{7}{6}} \;\big]\,\bigcup \, \big[ \sqrt{\dfrac{7}{6}} \;, \alpha^+\, \big); \\
&\\
\dfrac{2(2-\alpha^2)^{3/2}}{\sqrt{1+8(\alpha^2-1)^2-4(\alpha^2-1)^3}} \, ,& |\alpha| <  \sqrt{\dfrac{7}{6}}\, .
\end{array}
\right.
\end{equation}
\end{lemma}
\begin{proof}
We will estimate the maximum of the ratio
$$ S(x)/S(0) \ge g^2(x)/g^2(0) \ge \mu^2 .$$
The first case of (\ref{eqmu}) is just $q \ge 1/4$, where $ S(x)/S(0) \le 1$.
Suppose now that $q <1/4$ and
$$\alpha \in \big(\alpha^-, -\sqrt{\dfrac{7}{6}} \;\big]\,\bigcup \,\big[\sqrt{\dfrac{7}{6}} \;, \alpha^+\,\big).$$
Since the assumptions of Lemma \ref{sx0} are fulfilled for $x \in [0,x_0)$,
the maximum of $S(x)$ is attained at $x=x_0$ . Hence $S(x)/S(0)$ is bounded by
$$\frac{1+\varepsilon(0)}{1+\varepsilon(x_0)} = 1+\dfrac{(1-4q)^2(4-q)}{2 \left(25v(1 - q)^2 +8q^3-11q^2+4q-2\right)} \, ,$$
where $v=q u=\alpha^2-1$.
Here, as easy to check, for a fixed $v$ and $0 < q < 1/4$, the right hand side is decreasing in $q$. The choice $q=0$ gives
$$0 < \frac{1+\varepsilon(0)}{1+\varepsilon(x_0)} <1+ \frac{2}{25v-2}=\frac{\alpha^2-1}{\alpha^2-1.08} \, ,$$
and the result follows.

Next, let $|\alpha| < \sqrt{7/6} \, $. In this case
$$\frac{1+\varepsilon(0)}{1+\varepsilon(\sqrt{1-1/u}\;)} =\frac{2(1-q u)^3\left( 2 (1 - q)^2 u -1 + 2 q \right)}{(1-q)^2 (1 + u - 6 q u + 8 q^2 u^3 - 4 q^3 u^4)} <$$
$$\frac{2(2-\alpha^2)^3 u}{1-6 q u +(1 + 8 q^2 u^2 - 4 q^3 u^3)u} \le \frac{4(2-\alpha^2)^3}{1+8(\alpha^2-1)^2-4(\alpha^2-1)^3}\, .$$

\end{proof}

Upper bound on $\int_0^x|\varepsilon(t) b(t)|dt$ are given by the following lemma.

\begin{lemma}
\label{smallq}
\begin{equation}
\label{smallqeq}
\int_0^x |\varepsilon(t) b(t)| d t \le \left\{
\begin{array}{cc}
\dfrac{6x}{5 \sqrt{(1-x^2)u}} \, , & q < 0;\\
&\\
\dfrac{(1-x^2)x}{(1 - q) (1 - q - x^2)^{3/2} \sqrt{u}} \, , & 0 \le q <\frac{1}{2} \, ;\\
&\\
\dfrac{(1+q)x}{4\sqrt{u}\,(1 - q - x^2)^{3/2}} \, , &\frac{1}{2} \le q \le  1.
\end{array}
\right.
\end{equation}
\end{lemma}
\begin{proof}
 We have
$$-\varepsilon(x) b(x) = \frac{\omega(x)}{4\sqrt{u}\,(1-x^2)(1-q-x^2)^{5/2}}\,, $$
where
$$\omega(x)=x^6+6q x^4-3x^2+4q^2-6q+2  .$$
First we consider the case $|\alpha|<1$, that is $q < 0.$
Then $\omega (x) >0$ for $x\le \sqrt{1-1/u} \,.$ Indeed,
$$\frac{1}{6x} \, \omega'(x)=x^4+ 4 q x^2 -1= \frac{(4 \alpha^2-6)u +5-4 \alpha^2}{u^2}< \frac{(4\alpha^2-6)(u-1)}{u^2} <0. $$
Therefore by $u \ge 2$ for $k \ge 2$ and $|\alpha| \le 1,$  we have
$$u^3 \omega(x) \ge u^3 \omega (\sqrt{1-1/u} \; )=(19 - 20 \alpha^2 + 4 \alpha^4) u+6 \alpha^2-7 \ge 31 - 34 \alpha^2 + 8 \alpha^4 >0.$$
Using easy to check inequality
$$\omega(x)\le \frac{24}{5} \,(1-q-x^2)^2, \; \; q \le 0, \; x \le 1,$$
we obtain
$$\int_0^x |\varepsilon(t) b(t)| d t =\frac{1}{4 \sqrt{u}} \, \int_0^x \frac{\omega(t)}{(1-t^2)(1-q-t^2)^{5/2}}\, dt \le$$
$$=\frac{6}{5 \sqrt{u}} \, \int_0^x \frac{dt}{(1-t^2)\sqrt{1-q-t^2}}  \le \frac{6}{5 \sqrt{u}} \, \int_0^x \frac{dt}{(1-t^2)^{3/2}} = \frac{6x}{5 \sqrt{u (1-x^2)}} \, . $$



Let now $q \ge 1/2$, then $\omega(x) <0 $ for $x <\sqrt{1-q} \,$.
Indeed, replacing $x$ by $\sqrt{(1-q) \dfrac{s^2}{1+s^2} \, } \, $  and setting $q=\dfrac{1+p}{2}$, we get
$$- \frac{8 (1 + t^2)^3}{1-p} \, \omega(x) =5 (1 + p)^2 s^6+6(1+p)(3+p) s^4+12 (1 + 2 p) s^2+8p >0.$$
Hence in this case
$$4 \sqrt{u} \, \int_0^x |\varepsilon(t) b(t)| d t= -\int_0^x \frac{\omega(t)}{(1-t^2)(1-q-t^2)^{5/2}} \, dt = $$
$$\frac{ (3 - 3 q^2 - 3 x^2 + 2 q x^2)x}{3 (1 - q) (1 - q - x^2)^{3/2}} +
\arcsin \frac{x}{\sqrt{1 - q}} -\frac{4}{\sqrt{q}} \arcsin \frac{\sqrt{q} \, x}{\sqrt{(1-q)(1-x^2)}} \, .$$
Using $x \le \arcsin x \le \frac{\pi x}{2} \, , \; x \ge 0$, we convince that
$$\arcsin \frac{x}{\sqrt{1 - q}} -\frac{4}{\sqrt{q}} \arcsin \frac{\sqrt{q} \, x}{\sqrt{(1-q)(1-x^2)}} \le 0. $$
Therefore for $ \frac{1}{2}\le q < 1 $,
$$\int_0^x |\varepsilon(t) b(t)| d t \le  \frac{ (3 - 3 q^2 - 3 x^2 + 2 q x^2)x}{12\sqrt{u} \, (1 - q) (1 - q - x^2)^{3/2}} \le
\frac{(1+q)x}{4\sqrt{u}\,(1 - q - x^2)^{3/2}} \, .$$


Finally let $0 \le q \le 1/2$, then $\omega(x)$
can be written as a difference of two positive functions, $\omega(x)=\omega_1(x)-\omega_2(x)$,
$$\omega_1(x)=(1-q-x^2)(2 + 3 q - x^2 - 5 q x^2 - x^4), $$
$$\omega_2(x)=(7q-5q^2)(1-q-x^2)+5q^2(1-q),  $$
and thus
$|\omega(x)| \le
\omega_1(x) +\omega_2(x).
$

One finds
$$\int_0^x \frac{\omega_1(t)+\omega_2(t)}{ (1-t^2)(1-q-t^2)^{5/2}} \, dt =$$
$$ \frac{ (9 + 18 q - 27 q^2 - 9 x^2 - 22 q x^2)x}{3 (1 - q) (1 - q - x^2)^{3/2}}-\arcsin \frac{x}{\sqrt{1-q}}  \le $$
$$\frac{ (9 + 18 q - 27 q^2 - 9 x^2 - 22 q x^2)x}{3 (1 - q) (1 - q - x^2)^{3/2}} \le \frac{4(1-x^2)x}{ (1 - q) (1 - q - x^2)^{3/2}}\, ,$$
hence
$$\int_0^x |\varepsilon(t) b(t)| d t \le \frac{(1-x^2)x}{(1 - q) (1 - q - x^2)^{3/2} \sqrt{u}} \, .$$
This completes the proof.

\end{proof}


Now Lemmas \ref{mu} and \ref{smallq} infer the following claim which completes the proof of Theorem \ref{thmain1}.
\begin{lemma}
\begin{equation}
 |r(x)| < \left\{
\begin{array}{cc}
\dfrac{2.72 x}{\sqrt{(1-x^2)u}} \, , & q < 0 ;\\
   &\\
 \dfrac{2(1-x^2)x}{(1-q)(1-q-x^2)^{3/2} \, \sqrt{u}}\, , &  0 \le q <\frac{1}{2} \, ;\\
&\\
\dfrac{(1+q)x}{4\sqrt{u}\,(1 - q - x^2)^{3/2}} \, ,&\frac{1}{2} \le q < 1 \,.
\end{array}
\right.
\end{equation}
\end{lemma}
\begin{proof}
We have to match the bounds of Lemmas \ref{mu} and \ref{smallq}.
For $q <0$, that is for $|\alpha|<1$, we have
$$\mu \le \dfrac{2(2-\alpha^2)^{3/2}}{\sqrt{1+8(\alpha^2-1)^2-4(\alpha^2-1)^3}}\, .$$
The maximum of the last expression, even in a lager interval $|\alpha| \le \sqrt{7/6} \, $, is less than $34/15$ and attained for $\alpha=\sqrt{\frac{6-\sqrt{19}}{2}}\,$.
Thus, the numerical coefficient in this case is less than $\frac{34}{15} \cdot \frac{6}{5}=2.72$.
For $0 \le q <1/2$ we have to take the maximum of the bounds on $\mu$ for $|\alpha| \ge 1$. This yields $\mu \le 2$ corresponding to $\alpha =1$.
In the last case $q \ge 1/2$ we have $\mu \le 1.$
\end{proof}

\begin{remark}
\label{clxmax}
The obtained bounds on the error term $r(x)$ remain meaningful in a substantial part of the interval $|x| \le \sqrt{1-q} \, $.
If we set $x=\sqrt{(1-q)(1-\delta)}$, then $|r(x)| < c \, k^{-1} \delta^{-3/2}$, where one can take e.g., $c = 2 \frac{5}{9}\,$ for $0 \le q<1/2$, and  $c = 1/4$ for $q \ge 1/2$.
 Sharp upper and lower bounds on the extreme zeros of the Jacobi polynomials were given in \cite{krasup} and \cite{kraslow} respectively. In the ultraspherical case for the largest zero $x_{max}$ they are simplified to
 $$
x_{max}=\sqrt{\frac{k(k+2 \alpha+1)}{(k+\alpha+1)^2-k}} \, - \frac{3 (\alpha+1)^{4/3}(1+2\theta)}{2 k^{1/6}(k+2\alpha+1)^{1/6} ((k+\alpha+1)^2-k)^{5/6}}\, ,
$$
where $0<\theta<1$, $\alpha >-1, \; k \ge 5.$
This implies
$$x_{max} <\sqrt{\frac{k(k+2 \alpha+1)}{(k+\alpha+1)^2-k}} \, \left(1 - \frac{3 (\alpha+1)^{4/3}}{2k^{2/3}(k+\alpha+1)^{2/3}(k+2\alpha+1)^{2/3}} \right),$$
where the first factor does not exceed $\sqrt{1-q}$, whereas the second one is less than
$$1- \frac{3 \alpha^{4/3}}{2^{5/3}  (k+\alpha)^{4/3}k^{2/3}}\,  . $$
Thus, we get
$$x_{max}^2 < (1-q)\left( 1-\frac{3 \alpha^{4/3}}{2  (k+\alpha)^{4/3}k^{2/3}} \right), $$
what together with (\ref{urrx}) readily yields $|r(x_{max})| =O(1)$.

\end{remark}


\begin{remark}
In principle, the bounds on the error term can be strengthen by the iterative substitution of (\ref{eqg}) instead of $g(x)$ into (\ref{urr}). In particular, for large $u$, a lower bound on $r(x)$ can be obtained by estimating the following integral
$$\left|\,\int_0^x \varepsilon(t) b(t) \cos\mathcal{B}(t) \sin \left[ \mathcal{B}(t)-\mathcal{B}(x) \right]  dt \, \right| \ge $$
$$\frac{1}{2} \, \left|\, \sin\mathcal{B}(x) \, \int_0^x \varepsilon(t) b(t) dt \, \right|-
\frac{1}{2} \, \left|\, \int_0^x \varepsilon(t) b(t)\sin\left[ 2\mathcal{B}(t)-\mathcal{B}(x) \right] dt \,\right|.$$
The first integral in a closed form is
$$\int_0^x \varepsilon(t) b(t) dt=$$
$$  \frac{(3 - 3 q^2 - 3 x^2 + 2 q x^2)x}{12(1-q)(1 - q - x^2)^{3/2}\sqrt{u}}+
\frac{1}{4\sqrt{u}} \, \arcsin \frac{x}{\sqrt{1-q}}-\frac{1}{\sqrt{q u}} \, \arcsin \frac{\sqrt{q} \, x}{\sqrt{(1-q)(1-x^2)}} \, .$$
The second one contains a highly oscillating function and probably is negligible in comparison with the first. Thus it seems reasonable to conjecture that the error bounds of Theorem \ref{thmain1} are of the right order. In fact the situation is slightly more subtle, because for $q<1/2$ the main term changes the sign in the interval $0 <x <\sqrt{1-q} \,$.

\end{remark}

\begin{remark}
It would be very interesting to get a uniform bounds on ultraspherical polynomials in the transition region. An analogue of (\ref{resh}) (with the cosine replaced by Bessel functions) is readily available (see \cite{krasbes2}). However it is unclear how one can fix the constants of integration similar to $M$ and $\gamma$ above. The same problem (but seemingly in a less severe form) arises if we try to extend the results from the ultraspherical case to the general Jacobi polynomials. On the other hand, an amusing feature of the method we have used in this paper is that one does not need to know the constants of integration to estimate the relative error $r(x)$.
\end{remark}

\section{Proof of Theorem \ref{sharp}}
We will use the following inequality for the (continuous) central binomial coefficients:
\begin{equation}
\label{binom}
{2x \choose x}=\frac{\Gamma(2x+1)}{\Gamma^2 (x+1)} > \dfrac{4^x}{\sqrt{\pi (x+\frac{1}{2})}}\, , \; \; \; x \in \mathbb{R}^+.
\end{equation}
It is a direct consequence of the following:
the function $w(x)=4^{-x}{2x \choose x}\sqrt{\pi (x+\frac{1}{2})}$ is decreasing in $x$ for $x>0$, and, as Stirling's approximation shows, tends to one as $x \rightarrow \infty$.
To check that the function is decreasing, one finds
$$\frac{1}{2} \, (\log w(x))'=\psi(2x+1)-\psi(x+1)-\ln 2+\frac{1}{4x+2} \,,$$
where the polygamma function $\psi(x)$ satisfies the inequalities
$$ \ln x-\frac{1}{2x}-\frac{1}{12x^2} < \psi(x) < \ln x-\frac{1}{2x} \,, \; \; \; x>0,$$
(see \cite{AQ} and \cite{EL} for the lower and upper bound, respectively). These imply that $(\log w(x))' <0$, we omit the details.

\begin{proof}[Proof of Theorem \ref{sharp}]
By Theorem \ref{thmain1} we have
$$
\mathcal{J}=g^{-2}(0) \sqrt{u} \,\int_0^\eta (1-x^2)^\alpha \left(P_k^{(\alpha, \alpha)} (x) \right)^2 dx =g^{-2}(0)\int_0^\eta \frac{g^2(x)\, dx}{\sqrt{1-q-x^2}} =$$
$$ \int_0^\eta \frac{\cos^2 \mathcal{B}(x)\, dx}{\sqrt{1-q-x^2}}  +2 \int_0^\eta \frac{r(x) \cos \mathcal{B}(x)\, dx}{\sqrt{1-q-x^2}}  +
\int_0^\eta \frac{r^2(x)\, dx}{\sqrt{1-q-x^2}} :=$$
$$\mathcal{J}_1(x)+2\mathcal{J}_2(x)+\mathcal{J}_3(x) \ge
\mathcal{J}_1(x)-2 \left|\mathcal{J}_2(x)\right| .$$
Using $\mathcal{B}'(x)=b(x)$ and integrating by parts, we obtain
$$\mathcal{J}_1(\eta)=\frac{1}{2} \int_0^\eta \frac{dx}{\sqrt{1-q-x^2}}+\frac{1}{2}\int_0^\eta \frac{\cos 2 \mathcal{B}(x)}{\sqrt{1-q-x^2}}\, dx=$$
$$
\frac{1}{2} \,\arctan \frac{\eta}{\sqrt{1-q-\eta^2}} + \frac{1}{4\sqrt{u}} \, \int_0^\eta  \frac{(1-x^2) }{1-q-x^2} \,d \sin 2 \mathcal{B}(x) = $$

$$
\frac{1}{2} \, \arctan \frac{\eta}{\sqrt{1-q-\eta^2}} +\frac{(1-\eta^2)\sin 2 \mathcal{B}(\eta)}{4\sqrt{u}\;(1-q-\eta^2)}- \frac{q}{2\sqrt{u}} \,
\int_0^\eta  \frac{x \sin 2 \mathcal{B}(x)}{(1-q-x^2)^2} \; d x ,$$
where
$$\left|\int_0^\eta  \frac{x \sin 2 \mathcal{B}(x)}{(1-q-x^2)^2} \; d x \right| \le \int_0^\eta  \frac{x d x}{(1-q-x^2)^2}=\frac{\eta^2}{2(1-q)(1-q-\eta^2)}\, . $$
This along with the inequality
$$\arctan x >  \frac{\pi}{2}-\frac{1}{x} \, , \; \; \; x \ge 0 ,$$
implies
$$\mathcal{J}_1(\eta) >\frac{\pi}{4}-\frac{\sqrt{1-q-\eta^2}}{2\eta} -\frac{1-q-\eta^2+2q \,\eta^2}{4  (1-q)(1-q-\eta^2)\sqrt{u}} \, .  $$
A straightforward bound on $\left|\mathcal{J}_2(x) \right|$ is
$$\left|\mathcal{J}_2(x) \right| \le \int_0^\eta  \frac{|r(x)| \, dx}{\sqrt{1-q-x^2}} =$$
$$ \left\{
\begin{array}{cc}
\dfrac{q \eta^2}{(1-q)^2(1-q-\eta^2) \sqrt{u}} -\dfrac{\ln \dfrac{1-q}{1-q-\eta^2}}{(1-q)\sqrt{u}} \, , & 0 \le q < \frac{1}{2} \, ;\\
&\\
\dfrac{(1+q)\eta^2}{8(1-q)(1-q-\eta^2)\sqrt{u}} \, , & \frac{1}{2} \le q <1.
\end{array}
\right.
$$
Thus, in either case we can take
$$\left|\mathcal{J}_2(x) \right| \le \dfrac{ \eta^2}{(1-q)(1-q-\eta^2) \sqrt{u}} \, .$$
Since $\eta <\sqrt{1-q}$, this yields
$$\mathcal{J} >\frac{\pi}{4}-\frac{\sqrt{1-q-\eta^2}}{2\eta}- \frac{1-q+7\eta^2+2\eta^2 q}{4(1-q)(1-q-\eta^2) \sqrt{u}} >$$
$$\frac{\pi}{4}-\frac{\sqrt{1-q-\eta^2}}{2\eta}- \frac{2}{(1-q-\eta^2) \sqrt{u}} \, .$$
Let $\delta =\dfrac{4 \cdot 2^{1/3}}{3(1-q)^{2/3} u^{1/3}} \, ,$ then $\eta=\sqrt{1-q} \, (1-\delta) $, and noticing that
$\delta <1/2$ for $k \ge 10$ and $\alpha \ge 3-k/2$,
one finds
$$\frac{\sqrt{(2-\delta)\delta}}{2(1-\delta)} + \frac{2}{(2-\delta)\delta(1-q)\sqrt{u}} <\sqrt{\frac{3}{2} \, \delta}+ \frac{4}{3\delta(1-q)\sqrt{u}}=
\frac{3 }{2^{1/3}(1-q)^{1/3} u^{1/6}} \, .$$
Let
$$L=\frac{2^{2\alpha+1} \,\Gamma^2(k+\alpha+1)}{(2k+2\alpha+1)\Gamma(k+2\alpha+1)k!} \, , $$
that is $L$ is
$\big| \big|P_k^{(\alpha, \, \alpha)} \big| \big|_{L_2}^2 $, provided  $\alpha >-1$.

For the orthonormal normalization we obtain
$$\mathcal{I}=\int\limits_{- \, \sqrt{1-q}}^{\sqrt{1-q}} (1-x^2)^\alpha \left({\bf P}_k^{(\alpha, \alpha)} (x) \right)^2 dx >
\frac{g^2(0)}{\sqrt{u} \, L} \left(\frac{\pi}{2}-\frac{3 }{2^{1/3}(1-q)^{1/3}\, u^{1/6}} \right)= $$
$$\frac{\pi {k+\alpha \choose k/2}^2 \sqrt{k^2+2k \alpha+k+\alpha+1}}{2^{2k+1} \sqrt{u} \, L } \left(1- \frac{3 \cdot 2^{2/3} }{ \pi (1-q)^{1/3}\, u^{1/6}}  \,\right), $$
where by (\ref{binom}),
$$\frac{\pi {k+\alpha \choose k/2}^2 \sqrt{k^2+2k \alpha+k+\alpha+1}}{2^{2k+1}  \sqrt{u} \, L } =$$
$$\frac{\pi (2k+2 \alpha+1) \sqrt{k^2+2k \alpha+k+\alpha+1}}{4^{k+\alpha+1} \sqrt{(k+\alpha)(k+\alpha+1)}} \,{k \choose k/2} {k+2\alpha \choose k/2+\alpha} >$$
$$\frac{2k+2\alpha+1}{2} \, \sqrt{\frac{k^2+2k \alpha+k+\alpha+1}{(k+1)(k+\alpha)(k+\alpha+1)(k+2\alpha+1)}} >$$
$$\sqrt{\frac{k^2+2k \alpha+k+\alpha+1}{(k+1)(k+2\alpha+1)}} =\dfrac{1}{\sqrt{1+\dfrac{\sqrt{4u+1}-1}{2(1-q)u}}}> 1-\dfrac{1}{2(1-q) \sqrt{u}} \, .$$
Setting $z=(1-q)^{1/3} u^{1/6}$ and noticing that $z^2 >4$ for $k \ge 10, \; \alpha \ge -k/2+3$,  we conclude
$$\mathcal{I} >\left(1-\frac{1}{2z^3} \right) \left(1-\frac{3 \cdot 2^{2/3} }{ \pi z} \right) >1-\frac{5}{3z}>1-\frac{5}{3} \cdot \left( \dfrac{k+\alpha}{(k+2\alpha)k}\right)^{1/3} .$$
This completes the proof.

\end{proof}

\end{document}